\documentclass{amsart}
\usepackage{amsmath,amssymb}
\usepackage[pdftex]{graphicx}
\usepackage{braket}
\usepackage{color}

\title[Correction terms of $(+1)$-surgeries along $(2,q)$-cablings]{Heegaard Floer correction terms of $(+1)$-surgeries along $(2,q)$-cablings}
\author{Kouki Sato}
\date{}

\newtheorem{dfn}{Definition}
\newtheorem{question}{Question}
\newtheorem{thm}{Theorem}
\newtheorem{prop}{Proposition}
\newtheorem{lem}{Lemma}
\newtheorem{cor}{Corollary}

\DeclareMathOperator{\punc}{punc}
\DeclareMathOperator{\image}{Im}

\begin{document}
\maketitle

\begin{abstract}
The Heegaard Floer correction term ($d$-invariant) is an invariant of rational homology 3-spheres equipped with a Spin$^c$ structure. In particular, the correction term of  1-surgeries along knots in $S^3$ is a ($2\mathbb{Z}$-valued) knot concordance invariant $d_1$. 
In this paper, we estimate $d_1$ for the $(2,q)$-cable of any knot $K$. This estimate does not depend on the knot type of $K$. If $K$ belongs to a certain class which contains all 
negative knots, then equality holds. As a corollary, we show that the relationship between $d_1$ and the Heegaard Floer $\tau$-invariant is very weak in general.
\end{abstract}

\section{Introduction}

Throughout this paper we work in the smooth category,
all manifolds are compact, orientable and oriented.
If $X$ is a closed 4-manifold, then $\punc X$ denotes $X$ with an open 4-ball
deleted.

\subsection{Correction term and (2,q)-cablings}

In \cite{ozsvath-szabo},
 Ozsv\'{a}th and Szab\'{o} introduced a rational homology
cobordism invariant $d$ for rational homology 3-spheres equipped with
a Spin$^c$ structure from Heegaard Floer homology theory.
Here rational homology cobordism is defined as follows.
\begin{dfn}
\normalfont
For two rational homology 3-spheres $Y_i$
with Spin$^c$ structure $\frak{t_i}$ $(i = 1,2)$,
we say that $(Y_1, \frak{t}_1)$ is  
{\it rational homology cobordant to }$(Y_2, \frak{t}_2)$
if there exists an oriented cobordism
$W$ from $Y_1$ to $Y_2$ with $H_j (W; \mathbb{Q})=0$ ($j=1,2$)
which can be endowed with a Spin$^c$ structure $\frak{s}$
whose restrictions to the $Y_i$ are the $\frak{t}_i$  ($i=1,2$). 
\end{dfn}
This relation is an equivalence relation on the set of pairs $(Y, \frak{t})$
where $Y$ is a rational homology 3-sphere and $\frak{t}$ is a Spin$^c$
structure on $Y$. Moreover, the connected sum operation endows the 
quotient set $\theta^c$ of this relation
with the structure of an abelian group.

The invariant $d$ is called the {\it correction term}.
In particular, $d$ is a group homomorphism from $\theta^c$ to $\mathbb{Q}$.
Note that if $Y$ is an integer homology 3-sphere,
then $Y$ has a unique Spin$^c$ structure.
Hence in this case, we may denote the correction term simply by $d(Y)$ and 
it is known that the value of the invariant becomes an even integer.

Here we remark that for the integer homology 3-sphere $S^3_1(K)$ obtained by $(+1)$-surgery along a knot $K$, $d(S^3_1(K))$ is not only a rational homology cobordism invariant of 
$S^3_1(K)$, but also a knot concordance invariant of $K$.
In fact, Gordon \cite{gordon} proved that if two knots $K_1$ and $K_2$ are concordant,
then $S^3_1(K_1)$ and $S^3_1(K_2)$ are integer homology cobordant,
and this implies that $d(S^3_1(K_1)) = d(S^3_1(K_2))$.  
In the rest of the paper we denote $d(S^3_1(K))$ simply by $d_1(K)$ and investigate
$d_1$ as a knot concordance invariant.
Note that $d_1$ is a map from the knot concordance group to $2\mathbb{Z}$,
but not a group homomorphism.

While explicit formulas for some knot classes have been given
(for instance, alternating knots \cite{ozsvath-szabo2} and
torus knots \cite{borodzik-nemethi}),
calculating $d_1$ is difficult in general.
The calculation of $d_1$ is studied in \cite{peters}.
In this paper, we investigate $d_1$ of the $(2,q)$-cabling $K_{2,q}$
of an arbitrary knot $K$ for an odd integer $q > 1$.
In particular, we give the following estimate of $d_1(K_{2,q})$. 
\begin{thm}
\label{thm1}
For any knot $K$ in $S^3$ and $k \in \mathbb{N}$, we have
$$
d_1(K_{2,4k \pm 1}) \leq -2k.
$$
Moreover, if $K$ bounds a null-homologous disk
in $\punc (n\overline{\mathbb{C}P^2})$ for some $n \in \mathbb{N} $,
then this inequality becomes equality.
\end{thm}
We note that $d_1$ of the $(2,4k \pm 1)$-torus knot $T_{2,4k \pm 1}$ is
equal to $-2k$ \cite{borodzik-nemethi, ozsvath-szabo2},
and hence Theorem $\ref{thm1}$ implies that $d_1(K_{2,q}) \leq d_1(O_{2,q})$,
where $O$ is the unknot.
If $q$ is an odd integer with $q \leq 1$, then from the Skein inequality \cite[Theorem 1.4]{peters}
we have $-2 \leq d_1(K_{2,q}) \leq 0$, and so
$d_1(K_{2,q})$ is either $-2$ or $0$.
In this paper, we focus on the case where $q >1$.

Next, we consider knots which bound null-homologous disks in 
$\punc (n \overline{{\Bbb C}P^2})$.
We first assume that $K$ is obtained from the unknot by a sequence of isotopies and
crossing changes from positive to negative as in Figure \ref{crossing change}.
In this case, since such crossing changes can be realized by attaching
4-dimensional $(-1)$-framed 2-handles to $S^3$ (giving rise to the 
$\overline{\mathbb{C}P^2}$ factors)
and handle slides (yielding the capping surface) as in Figure \ref{crossing change handle},
we have the desired disk with boundary $K$.
(The disk is null-homologous because it is obtained from the initial capping disk
of the unknot by adding, as a boundary connected sum, two copies of the core of each
2-handle with opposite sign.)
This implies that our knot class contains any negative knot.
On the other hand, if the Heegaard Floer $\tau$-invariant $\tau(K)$
of $K$ is more than 0, then $K$ cannot bound such a disk in $\punc (n\overline{\mathbb{C}P^2})$. This follows immediately from 
\cite[Theorem 1.1]{ozsvath-szabo3}.

\subsection{Comparison with the $\tau$-invariant}
The {\it Heegaard Floer $\tau$-invariant $\tau$} is
a knot concordance invariant defined by Ozsv\'{a}th-Szab\'{o} 
\cite{ozsvath-szabo3} and Rasmussen \cite{rasmussen}. 
In comparing the computation of $d_1$ and $\tau$,
Peters poses the following question.

\begin{question}[Peters \cite{peters}]
What is the  relation between $d_1(K)$ and $\tau(K)$?
Is it necessarily true that
$$
|d_1(K)| \leq 2|\tau(K)| ?
$$ 
\end{question}

Krcatovich \cite{krcatovich} has already given a negative answer to this question.
In fact, he showed that for any positive even integer $a$, there exists
a knot $K$ which satisfies $\tau(K) = 0$ and $|d_1(K)| = a$.
In this paper, we give a stronger negative answer.

\begin{thm}
\label{thm2}
For any even integers $a$ and $b$ with $a > b \geq 0$, there exist 
infinitely many knot concordance classes $\{ [K^n] \}_{n \in {\Bbb N}}$
such that for any $n \in {\Bbb N}$,
$$
|d_1(K^n)|= a \text{ and } 2|\tau(K^n)|=b.
$$
\end{thm}

This theorem follows from Theorem \ref{thm1} and the following theorem by Hom.
These two theorems give the following contrast between $d_1$ and $\tau$: 
for certain knots,
$\tau(K_{2,q})$ depends on the choice of $K$, while $d_1(K_{2,q})$ does not depend.

\begin{thm}[Hom \cite{hom}]
\label{thm.hom}
Let $K$ be a knot in $S^3$ and $p>1$. Then $\tau(K_{p,q})$
is determined in the following manner.
\begin{enumerate}
\item If $\varepsilon(K)=1$, then $\tau(K_{p,q}) = p\tau(K) + (p-1)(q-1)/2$.
\item If $\varepsilon(K)=-1$, then $\tau(K_{p,q}) = p\tau(K) + (p-1)(q+1)/2$.
\item If $\varepsilon(K)=0$, then $\tau(K_{p,q}) = \tau(T_{p,q})$.
\end{enumerate}
Here $\varepsilon(K) \in \set{0, \pm1}$ is a knot concordance invariant of $K$
defined in \cite{hom}.
\end{thm}

\begin{figure}[tbp]
\begin{center}
\includegraphics[scale = 0.8]{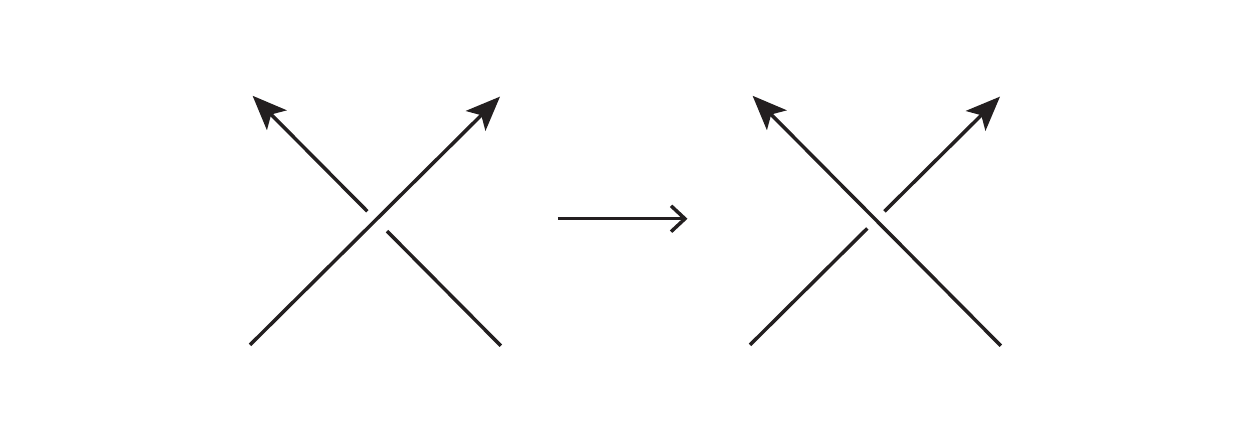}
\vspace{-8mm}
\caption{\label{crossing change}}
\includegraphics[ scale = 0.8]{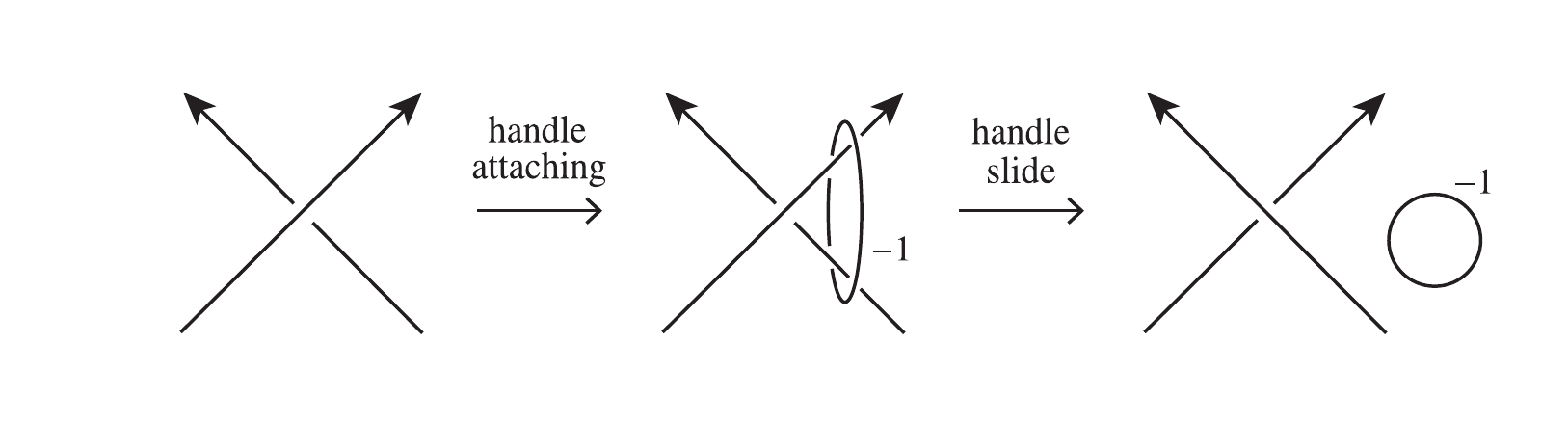}
\vspace{-12mm}
\caption{\label{crossing change handle}
}
\end{center}
\end{figure}

\subsection{An idea of proofs and another application}

In this subsection, we describe how we obtain an estimate of $d_1(K_{2,q})$.
We recall that if an integer homology 3-sphere $Y$ bounds a negative definite 4-manifold,
then $d(Y)$ satisfies the following inequality.
\begin{thm}[Ozsv\'{a}th-Szab\'{o} \cite{ozsvath-szabo}]
\label{thm.o-s}
Let $Y$ be an integer homology 3-sphere, then for each negative definite
4-manifold $X$ with boundary $Y$, we have the inequality
$$
Q_X(\xi, \xi) + \beta_2(X) \leq 4d(Y)
$$
for each characteristic vector $\xi$.
\end{thm}
Here a 4-manifold $X$ is called {\it negative (resp.\ positive) definite}
if the intersection form $Q_X$ of $X$ is negative (resp.\ positive) definite.
Moreover, $\xi \in H_2(X; \mathbb{Z})$ is called a {\it characteristic vector}
if $\xi$ satisfies $Q_X(\xi , v) \equiv Q_X(v,v) \mod 2$
for any $v \in H_2(X; \mathbb{Z})$,
and
$\beta_i$ denotes the $i$-th Betti number.

Let us also recall the definition of $d$.
For a rational homology 3-sphere $Y$ with Spin$^c$ structure $\frak{t}$,
the Heegaard Floer homologies $HF^*(Y, \frak{t})$ ($* = +,-,\infty$) are defined
as an absolute $\mathbb{Q}$-graded $\mathbb{Z}[U]$-modules,
where the action of $U$ decreases the grading by $2$.
These homology groups are related to one another by an exact
sequence:
$$\cdots \to HF_{\bullet}^-(Y,\frak{t})\to HF_{\bullet}^\infty(Y,\frak{t})\overset{\pi}{\to} HF_{\bullet}^+(Y,\frak{t})\to HF_{\bullet-1}^-(Y,\frak{t})\to \cdots$$
Then $\pi (HF^\infty(Y,\frak{t})) \subset HF^+(Y,\frak{t})$ is isomorphic to 
$\mathbb{Z}[U,U^{-1}]/U \cdot \mathbb{Z}[U]$, and so
we can define $d(Y, \frak{t})$ to be the minimal grading of $\pi (HF^\infty(Y,\frak{t}))$.
 
In order to prove our main results, we only use Theorem \ref{thm.o-s} without
investigation of any Heegaard Floer homology groups.
Indeed, the following theorem plays an essential role in this paper.
\begin{thm}
\label{thm3}
For any knot $K$ and positive integer $k$, 
there exists a 4-manifold $W$ which satisfies
\begin{enumerate}
\item $W$ is a simply-connected Spin manifold, 
\item $\partial W = S^3_1\left( K_{2,4k \pm 1} \right)$, and 
\item $\beta_2(W)=\beta^+_2(W)= 8k$. In pariticular, $W$ is positive definite.
\end{enumerate}
\end{thm}
Here $\beta^+_2(W)$ (resp.\ $\beta^-_2(W)$) denotes the number
of positive (resp.\ negative) eigenvalues of $Q_W$.
Theorem \ref{thm3} implies that for any knot $K$, $-S^3_1(K_{2, 4k\pm1})$ bounds a
negative definite Spin 4-manifold $-W$, and we obtain 
the inequality of Theorem \ref{thm1} by
applying Theorem \ref{thm.o-s} to the pair $(-W, -S^3_1(K_{2,4k \pm 1}))$.

We also mention that Tange \cite{tange} investigated 
which Brieskorn homology spheres bound a definite Spin 4-manifold, 
and also asked which integer homology 3-spheres bound a definite Spin 4-manifold.
Theorem \ref{thm3} gives a new construction of such 3-manifolds.
Furthermore,  for any knot $K$ which bounds a null-homologous disk in
$\punc (n\overline{\mathbb{C}P^2})$ for some $n$,
Theorems\ \ref{thm1}, \ref{thm.o-s} and \ref{thm3} let us determine 
the value of $\frak{ds}(S^3_1(K_{2,q}))$,
where $\frak{ds}$ is an $h$-cobordism invariant of integer homology 3-spheres defined by Tange \cite{tange} as follows:
$$
\frak{ds}(Y) = \max \Set{ \frac{\beta_2(X)}{8} | 
\begin{array}{l}
\partial X =Y\\
H_0(X)= \mathbb{Z}, H_1(X)=0 \\
\beta_2(X)=  \left| \sigma(X) \right|,  w_2(X) = 0 
\end{array}
}.
$$
\begin{cor}
\label{cor1}
If $K$ bounds a null-homologous disk in $\punc(n\overline{\mathbb{C}P^2})$, then 
$$\frak{ds}(S^3_1(K_{2, 4k \pm 1})) = k.$$
\end{cor}

In addition, we also give the following result.

\begin{prop}
\label{cor2}
For any positive integer $k$, 
there exist infinitely many irreducible integer homology 3-spheres
whose $\frak{ds}$ value is $k$.
\end{prop}

\subsection{Further questions}
It is natural to ask that if 
$K$ does not bound a null-homologous disk in $\punc(n\overline{\mathbb{C}P^2})$
for any $n$, then what are the possible $d_1(K_{2,q})$ values. 
As a case study on this question, we compute $d_1$ of
the $(2, 2pq \pm 1 )$-cabling of the $(p,q)$-torus knot for $p,q>0$,
which does not bound a null-homologous disk in $\punc(n\overline{\mathbb{C}P^2})$
for any $n$.
\begin{prop}
\label{prop}
For any positive coprime integers $p$ and $q$,  we have 
$$
d_1((T_{p,q})_{2, 2pq \pm 1}) = d_1(T_{2, 2pq \pm 1}).
$$
\end{prop}

This result is derived from \cite[Proposition 8.1]{ozsvath-szabo} and the fact that 
$(4pq \pm 1)$-surgery along $(T_{p,q})_{2,2pq \pm1}$ gives a lens space \cite{bleiler-litherland}. 
Based on this proposition and Theorem \ref{thm1}, we suggest the following question.

\begin{question}
Does the equality $d_1(K_{2,q}) = d_1(T_{2,q})$ hold for all knots $K$?
\end{question}

Finally we suggest a question related to Proposition \ref{cor2}.
While $\frak{ds}$ is an $h$-cobordism invariant, 
we will only prove that those infinitely many integer homology 3-spheres in
Proposition \ref{cor2} are not diffeomorphic to one another.
Whether they are $h$-cobordant or not remains open.
More generally, we suggest the following question.

\begin{question}
Is it true that for any positive integer $q$ and 
any two knots $K$ and $K'$,  $S^3_1(K_{2,q})$ and $S^3_1(K'_{2,q})$ are $h$-cobordant?
\end{question}

\subsection{Acknowledgements}
The author was supported by JSPS KAKENHI Grant Number 15J10597.
The author would like to thank his supervisor, Tam\'{a}s K\'{a}lm\'{a}n
for his useful comments and encouragement.

\section{Proof of Theorem \ref{thm3}}

In this section, we prove Theorem \ref{thm3}.
To prove the theorem, we first prove the following lemma.
Here we identify $H_2(\punc X, \partial (\punc X); \mathbb{Z})$
with $H_2(X; \mathbb{Z})$.
\begin{lem}
\label{lem1}
Let $X \cong S^2 \times S^2 \text{ or } \mathbb{C}P^2 \# \overline{\mathbb{C}P^2}$
and $K \subset \partial (\punc X)$ a knot.
If $K$ bounds a disk $ D \subset \punc X$
with self-intersection $-n<0$ 
which represents a characteristic vector in $H_2(X;\mathbb{Z})$,
then there exists a 4-manifold $W$ which satisfies
\begin{enumerate}
\item $W$ is a simply-connected Spin 4-manifold,
\item $\partial W = S^3_1\left( K \right)$, and
\item $\beta_2(W)=\beta^+_2(W)= n$. In particular, $W$ is negative definite．
\end{enumerate}
\end{lem}

\begin{proof}
By assumption, $K$ bounds a disk $D$ in $\punc X$
which represents a characteristic vector in $H_2(X;\mathbb{Z})$.
Since $[D, \partial D]$ is a characteristic vector and $\sigma(X) = 0$,
we have $-n=[D, \partial D] \cdot [D, \partial D] \in 8\mathbb{Z}$
(see \cite[Section 3]{kirby}).

By attaching $(+1)$-framed 2-handle $h^2$ along $K$, and gluing $D$ with
the core of $h^2$, we obtain an embedded 2-sphere $S$ in $W_1 := X \cup h^2$
such that $S$ represents a characteristic vector in $H_2(W_1; {\Bbb Z})$ and satisfies 
$[S] \cdot [S] = -n+1<0$.
We next take the connected sum 
$(W_2,S')= (W_1, S) \# ( \#_{n-2} (\mathbb{C}P^2,\mathbb{C}P^1))$,
and then $(W_2,S')$ satisfies
\begin{enumerate}
\item $\partial W_2 = S^3_1(K)$,
\item $\beta^+_2(W_2)=n$，$\beta^-_{2}(W_2)=1$, and 
\item $[S']$ is a characteristic vector and $[S']\cdot[S']=-1$.
\end{enumerate}
Therefore there exists a Spin 4-manifold $W$
which satisfies $W_2 = W \# \overline{CP^2}$.
We can easily verify that this 4-manifold $W$ satisfies the assertion of Lemma \ref{lem1}.
\end{proof}

We next prove the following lemma.

\begin{lem}
\label{lem2}
For any knot
$K \subset \partial(\punc (S^2 \times S^2))$ and $k \in \mathbb{N}$,
$K_{2, 4k \pm 1}$ bounds a disk $D$ in $\punc (S^2 \times S^2)$
with self-intersection $-8k$
which represents a characteristic vector.
\end{lem}

\begin{proof}
For a given $K_{2, 4k \pm 1}$,
we apply the hyperbolic transformation shown in Figure \ref{hyperbolic}
and Figure \ref{half} to $2k$ full twists 
and a $\pm 1$ half twist of $K_{2,4k \pm 1}$ respectively.
Then we have a concordance $A$ (with genus zero) in $S^3 \times [0,1]$ 
from $K_{2,4k \pm 1} \subset S^3 \times \{ 0 \}$ to the link $L \subset S^3 \times \{ 1 \}$
shown in Figure \ref{link}. 
Note that the 4-manifold $X$ shown in Figure \ref{X} is diffeomorphic to $S^2 \times S^2$,
and $L$ is the boundary of $2k+2$ disks $E$ in $X$ with the 0-handle deleted,
where $E$ consists of 2 copies of the core of $h^2_1$ 
and $2k$ copies of the core of $h^2_2$.
By gluing $(S^3 \times [0,1], A)$ and $(\punc X, E)$ along $(S^3, L)$,
we obtain a disk $D$ in $\punc X$ with boundary $K_{2, 4k \pm 1}$.
It is easy to see that $D$ represents a characteristic vector 
and has the self-intersection $-8k$.
\end{proof}

\begin{figure}[thbp]
\hspace{-8mm}
\includegraphics[scale = 0.7]{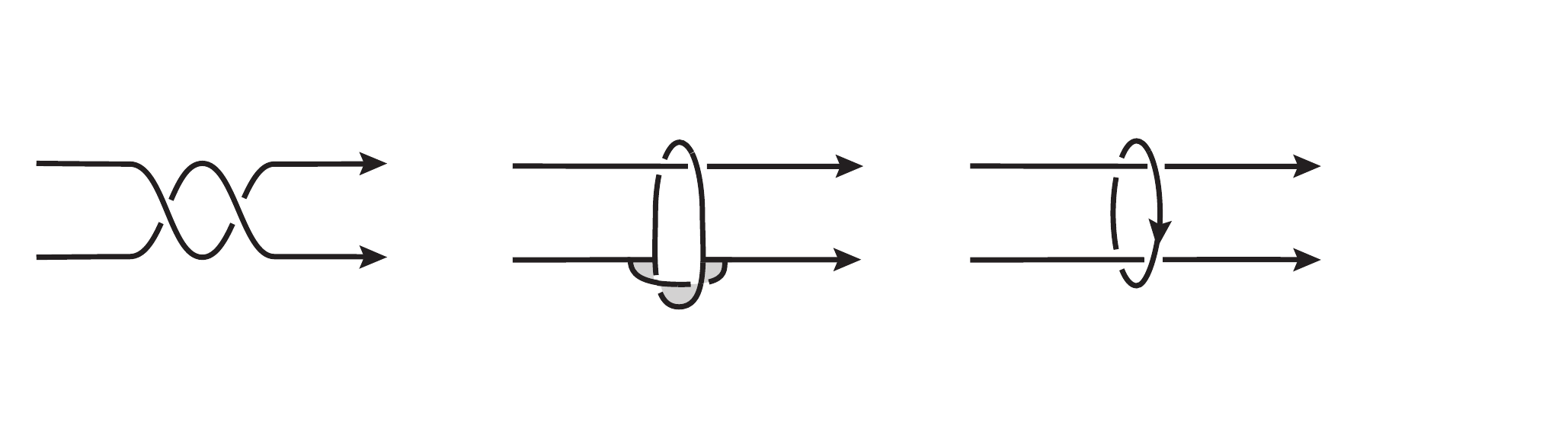}
\vspace{-14mm}
\caption{\label{hyperbolic}}
\hspace{-20mm}
\includegraphics[scale = 0.7]{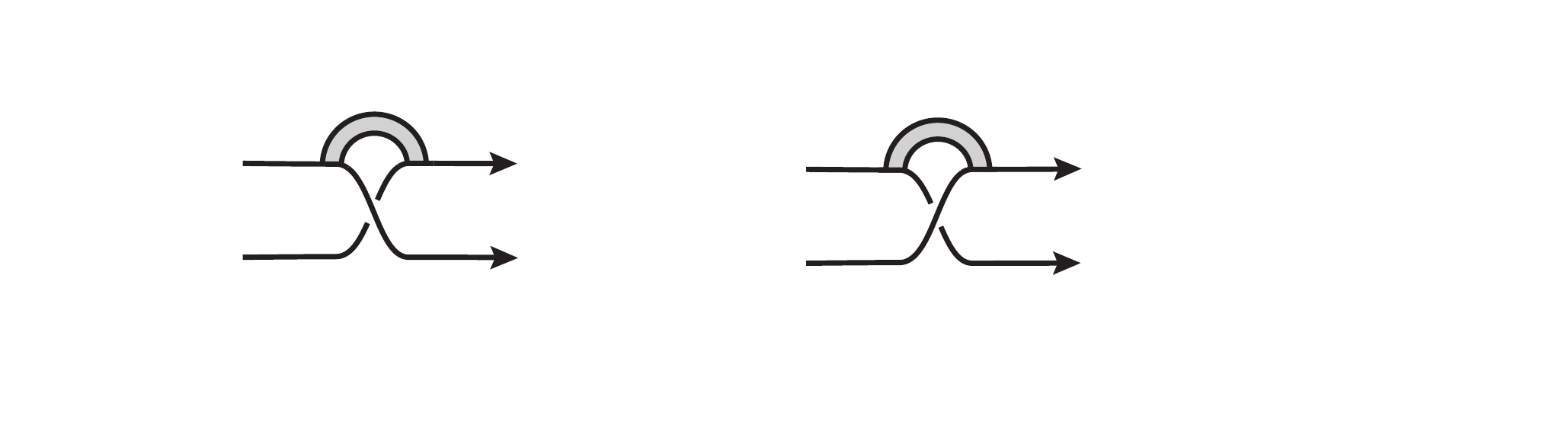}
\vspace{-16mm}
\caption{\label{half}}
\begin{minipage}[]{0.4\hsize}
\hspace{-8mm}
\includegraphics[scale = 0.65]{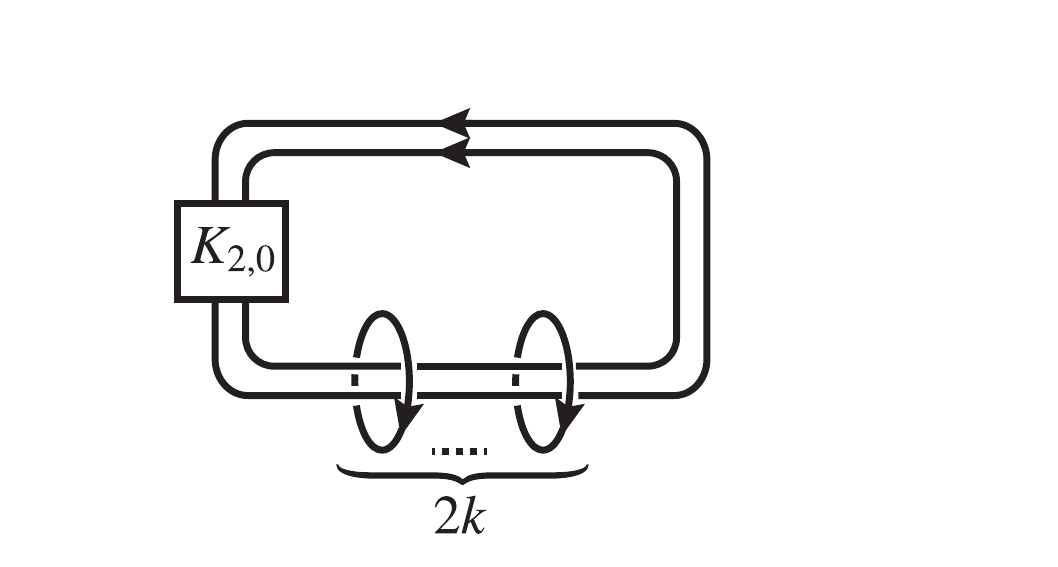}
\caption{\label{link}}
 \end{minipage}
 \begin{minipage}[]{0.4\hsize}
\hspace{-8mm}
\includegraphics[scale = 0.65]{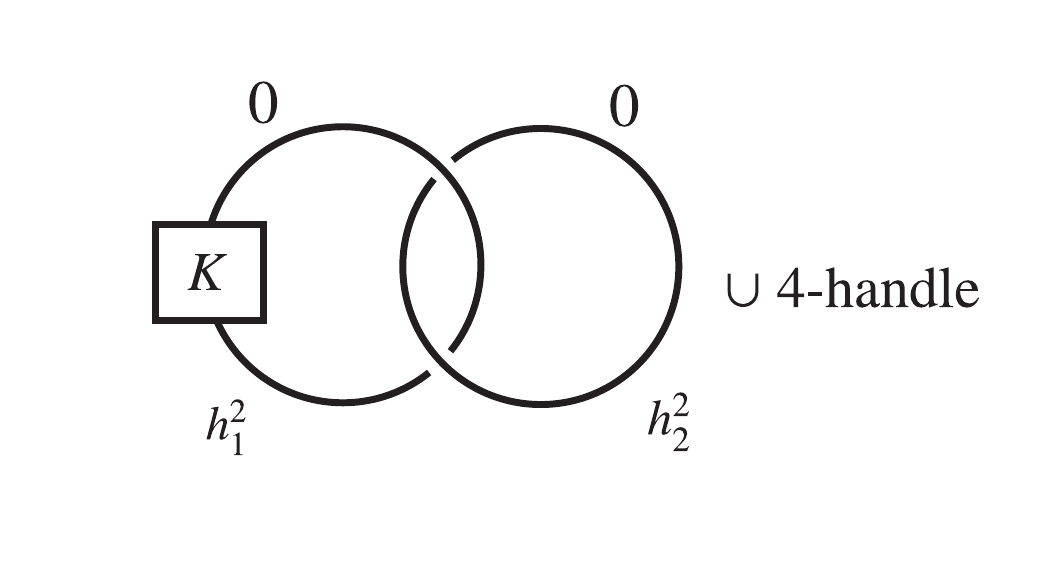}
\caption{\label{X}}
 \end{minipage}
\end{figure}

\def\proofname{Proof of Theorem \ref{thm3}}
\begin{proof}
By Lemma \ref{lem2}, $K_{2, 4k \pm 1}$ bounds a disk $D$ in $\punc (S^2 \times S^2)$
with self-intersection $-8k$ which represents a characteristic vector.
Then, by applying Lemma \ref{lem1} to the pair $(\punc (S^2 \times S^2), D)$,
we obtain the desired 4-manifold. 
\end{proof}

\section{Proof of Theorem \ref{thm1}, Theorem \ref{thm2}, and corollaries}

In this section, we prove Theorem \ref{thm1}, Theorem \ref{thm2}, and 
two corollaries.
We first prove Theorem \ref{thm1}, and then Corollary \ref{cor1} immediately
follows from Theorem \ref{thm1}.
\def\proofname{Proof of Theorem \ref{thm1}}
\begin{proof}
Let $W$ be a 4-manifold with boundary $S^3_1(K_{2,4k \pm 1})$ 
which satisfies the assertion of Theorem $\ref{thm3}$.
Since $W$ is a Spin 4-manifold, the trivial element of $H_2(W; \mathbb{Z})$
is a characteristic vector. By Theorem \ref{thm.o-s}, we have
$$
0+ 8k \leq 4d(-S^3_1(K_{4k \pm 1})).
$$
Since $d(-S^3_1(K_{4k \pm 1})) = - d(S^3_1(K_{4k \pm 1})) = -d_1(K_{4k \pm 1})$,
this gives the inequality in Theorem \ref{thm1}.

Next, we suppose that $K$ bounds a null-homologous disk $D$ in
$\punc \overline{\mathbb{C}P^2}$ for some $n \in \mathbb{N}$.
Excising a neighborhood of an interior point of $D$ in $\punc \overline{\mathbb{C}P^2}$,
we obtain a null-homologous annulus $A$ properly embedded in $\overline{\mathbb{C}P^2}$ with
the 0-handle $h^0$ and the 4-handle $h^4$ deleted such that
$(\partial h^0 , A \cap \partial h^0)$ is the unknot and 
$(\partial h^4 , A \cap \partial h^4)$ is $K$.
Furthermore, since $A$ is null-homologous, $A$ gives  a null-homologous annulus $A'$
in $\overline{\mathbb{C}P^2} \setminus (h^0 \cup h^4)$ such that 
$(\partial h^0 , A' \cap \partial h^0)$ is $T_{-2, 4k \pm 1}$ and 
$(\partial h^4 , A' \cap \partial h^4)$ is $K_{2, 4k \pm 1}$.
Now we attach a $(+1)$-framed 2-handle $h^2$ along $K_{2, 4k \pm 1}$,
and remove a neighborhood of a disk $D'$ from $\overline{\mathbb{C}P^2} \setminus (h^0 \cup h^4)$, where $D'$ is a disk obtained by gluing $A'$ with the core of $h^2$.
Then we have a negative definite 4-manifold $W$ with boundary 
$S^3_1(K_{2,4k \pm 1}) \amalg -S^3_{1}(T_{2, 4k \pm 1})$.
To see this, we regard a neighborhood of $D'$ as a $(+1)$-framed 2-handle along
the mirror image of $T_{-2,4k \pm 1}$, i.e., $T_{2, 4k \pm 1}$.
If we denote the union of $h^0$ and this 2-handle by $X$,
then $\overline{W} := (\overline{\mathbb{C}P^2} \setminus h^4) \cup h^2$ can be
regarded as the 4-manifold obtained by gluing $X$ to $W$ along $-S^3_1(T_{2,4k \pm 1})$
(see Figure \ref{cobordism}). In addition, the boundary of $\overline{W}$ is
$S^3_1(K_{2,4k \pm 1})$, hence the boundary of $W$ is the disjoint union of
$S^3_1(K_{2, 4k \pm 1})$ and $-S^3_1(T_{2,4k \pm 1})$.
The negative definiteness of $W$ follows from the fact
that the inclusion maps induces the isomorphism among the second homologies
$(i_X)_* + (i_W)_*: H_2(X; \mathbb{Z}) \oplus H_2(W; \mathbb{Z}) \cong H_2(\overline{W}; \mathbb{Z})$ 
and the intersection forms
$Q_X \oplus Q_W = Q_{\overline{W}}$, and 
$\beta^+_2(\overline{W}) = \beta^+_2(X) = 1$.

We apply Theorem \ref{thm.o-s} to the pair 
$(W, S^3_1(K_{2,4k \pm 1}) \amalg -S^3_1(T_{2,4k \pm 1}))$.
Let $\gamma \in H_2(\overline W; \mathbb{Z})$ be the generator of $h^2$
and $\overline{\gamma}_1, \ldots, \overline{\gamma}_n \in H_2(\overline{W}; \mathbb{Z})$
the generators induced from $H_2(n \overline{\mathbb{C}P^2} \setminus h^4; \mathbb{Z})$
which satisfy $Q_{\overline{W}}(\overline{\gamma}_i, \overline{\gamma}_j) = -\delta_{ij}$
(Kronecker's delta).
Then the tuple $\gamma,  \overline{\gamma}_1 , \ldots, \overline{\gamma}_n$ is a basis
of $H_2(\overline{W}; \mathbb{Z})$ and gives a representation matrix 
$$
Q_{\overline{W}} =
\left(
\begin{array}{cccc}
1&0&\cdots&0\\
0&-1&\ &0\\
\vdots&\ &\ddots &\vdots \\
0&0&\cdots&-1\\
\end{array}
\right).
$$
Moreover, we see $\image (i_X)_* = \mathbb{Z} \langle \gamma \rangle$,
and this gives 
$\image (i_W)_* = 
\mathbb{Z} \langle \overline{\gamma}_1, \ldots, \overline{\gamma}_n \rangle$. 
Hence we can identify $(H_2(W;\mathbb{Z}), Q_W)$ with
the pair of $\mathbb{Z} \langle \overline{\gamma}_1, \ldots, \overline{\gamma}_n \rangle$
and the intersection form 
$$
\left(
\begin{array}{ccc}
-1&\cdots &0\\
\vdots&\ddots &\vdots \\
0&\cdots&-1\\
\end{array}
\right).
$$
Now we apply Theorem \ref{thm.o-s} to the tuple 
$(W, S^3_1(K_{2,4k \pm 1}) \amalg -S^3_1(T_{2,4k \pm 1}), \sum_{i=1}^n \overline{\gamma}_i)$,
and we have the following inequality
\begin{equation}
\label{eq1}
\sum_{i=1}^n (-1) + n \leq 4d(S^3_1(K_{2,4k \pm 1})) -4d(S^3_1(T_{2,4k \pm 1})).
\end{equation}
Since $T_{2, 4k \pm 1}$ is an alternating knot, \cite[Corollary 1.5]{ozsvath-szabo2} gives 
the equality $d(S^3_1(T_{2,4k \pm 1})) = -2k$. 
This equality reduces $(\ref{eq1})$ to the inequality 
$$
-2k \leq d(S^3_1(K_{2,4k \pm 1})).
$$
This completes the proof.
\end{proof}

\begin{figure}[thbp]
\includegraphics[scale = 0.6]{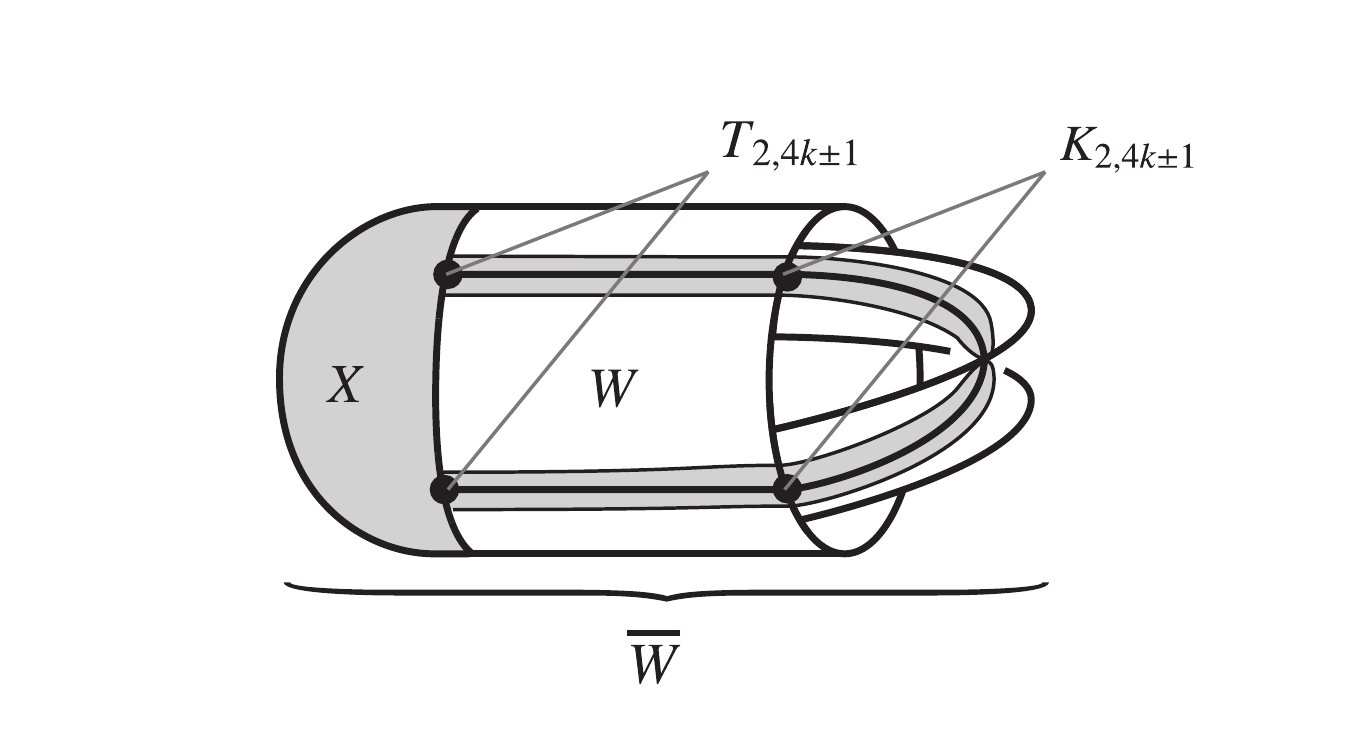}
\vspace{-4mm}
\caption{\label{cobordism}}
\end{figure}

\def\proofname{Proof of Corollary \ref{cor1}}
\begin{proof}
Let $K$ be a knot which bounds a null-homologous disk in
$\punc (n \overline{\mathbb{C}P^2})$.
By Theorem \ref{thm3}, $S^3_1(K_{2,4k \pm 1})$ bounds
a positive definite Spin 4-manifold which satisfies the conditions of $\frak{ds}$.
Hence we have $\frak{ds}(S^3_1(K_{2,4k \pm 1})) \geq k$.
Moreover, Theorem \ref{thm1} and \cite[Theorem 2.1(9)]{tange} give the inequality 
$$
\frak{ds}(S^3_1(K_{2,4k \pm 1})) \leq |d(S^3_1(K_{2,4k \pm 1}))|/2 = k.
$$
This completes the proof.
\end{proof}

We next prove Proposition \ref{cor2}.

\def\proofname{Proof of Proposition \ref{cor2}}
\begin{proof}
For any $n \in \mathbb{Z}$, let $K_n$ be the knot shown in Figure \ref{K_n}.
Since the cobordism in Figure \ref{K_n_cobordism} gives a null-homologous disk in $\punc \overline{\mathbb{C}P^2}$
with boundary $K_n$,  it follows from Corollary \ref{cor1} 
that the equality $\frak{ds}(S^3_1((K_n)_{2,4k + 1})) = k$ holds.
We denote $S^3_1((K_n)_{2,4k + 1})$ by $M_{n,k}$ and we will prove
that $M_{n,k}$ is irreducible and if $m \neq n$, then $M_{m,k}$ is not 
diffeomorphic to $M_{n,k}$.

We recall that the Casson invariant of $S^3_1(K)$, denoted by $\lambda(S^3_1(K))$,
is obtained from the following formula
$$
\lambda(S^3_1(K)) = \frac{1}{2} \Delta''_K (1),
$$
where $\Delta_K(t)$ is the normalized Alexander polynomial of $K$ such that
$\Delta_K(1) = 1$ and $\Delta_K(t)=\Delta_K(t^{-1})$ 
(see \cite[Theorem 3.1]{saveliev}).
It is easy to compute that for any knot $K$, 
$\Delta''_{K_{2,4k +1}}(1) = 
(\Delta_{K}(t^2) \cdot \Delta_{T_{2,4k +1}}(t))''|_{t=1} 
= 4\Delta''_{K}(1) + \Delta''_{T_{2,4k +1}}(1)$.
Furthermore, we can easily compute that
$\Delta_{K_n}(t)= nt -(2n-1) +nt^{-1}$ and 
$\Delta''_{K_n}(1) = 2n$.
These imply that if $m \neq n$, then 
$$
\lambda(M_{n,k}) - \lambda(M_{m,k}) = 2\Delta''_{K_n}(1) - 2\Delta''_{K_m}(1) = 4n -4m \neq 0,
$$
and hence $M_{m,k}$ is not diffeomorphic to $M_{n,k}$.

We next prove that $M_{n,k}$ is irreducible. 
The transformation shown in Figure \ref{tunnel} implies
that for any $n \in \mathbb{Z}$ and $k \in \mathbb{N}$, 
$(K_n)_{2,4k + 1}$ has the tunnel number at most 2,
and hence $M_{n,k}$ has the Heegaard genus at most 3.
By the additivity of the Heegaard genus,
if $M_{n,k}$ can be decomposed to $N_1$ and $N_2$, then
either $N_1$ or $N_2$ has the Heegaard genus 1.
Assume that $N_1$ has the Heegaard genus 1. 
Then $N_1$ is diffeomorphic to one of $S^3$, $S^1 \times S^2$, and lens spaces.
However, $M_{n,k}$ is an integer homology 3-sphere, and hence $N_1$ must be
diffeomorphic to $S^3$.
This completes the proof.
\end{proof}

\begin{figure}[thbp]
\hspace{-6mm}
\includegraphics[scale = 0.5]{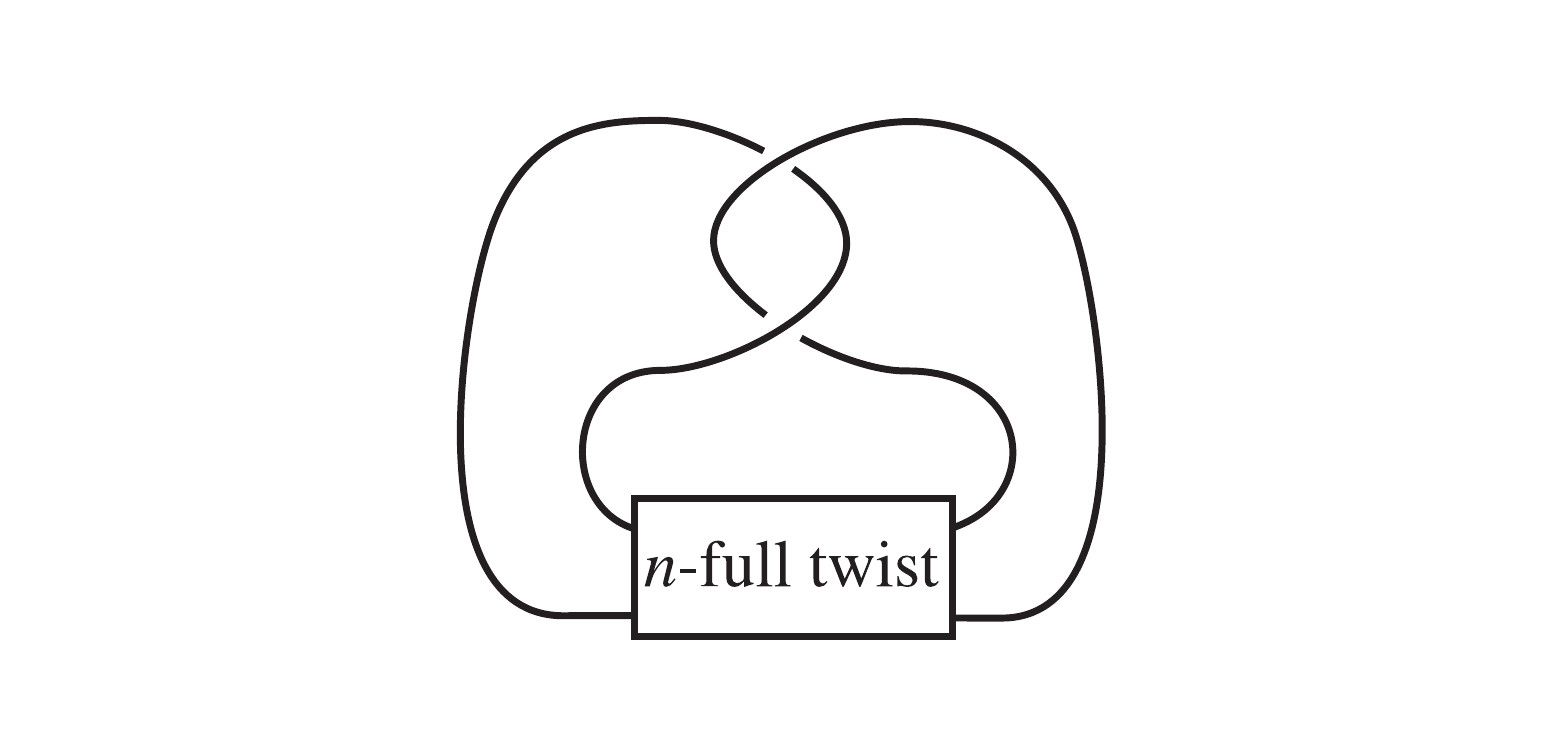}
\vspace{-6mm}
\caption{\label{K_n}}
\hspace{5mm}
\includegraphics[scale = 0.5]{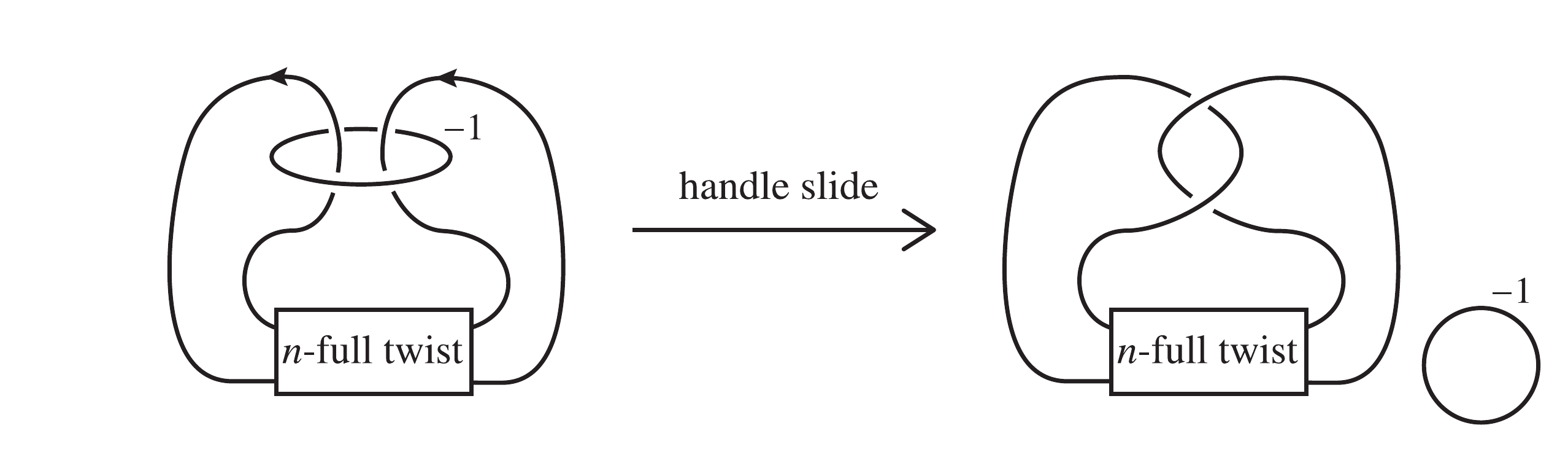}
\vspace{-10mm}
\caption{\label{K_n_cobordism}}
\hspace{-5mm}
\includegraphics[scale = 0.5]{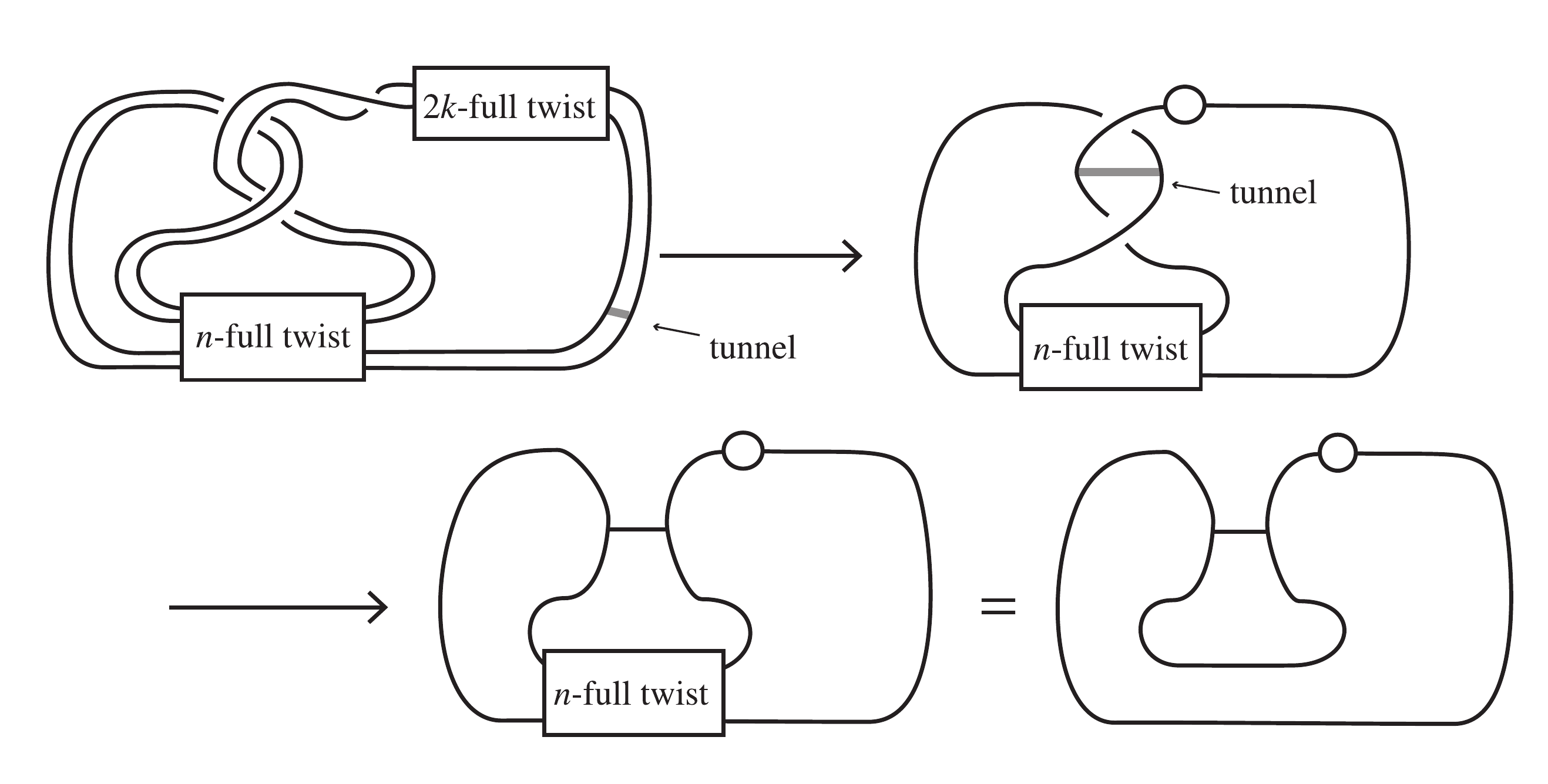}
\vspace{-8mm}
\caption{\label{tunnel}}
\end{figure}

Finally we prove Theorem \ref{thm2}.

\def\proofname{Proof of Theorem \ref{thm2}}
\begin{proof}
Suppose that $K$ is a negative knot.
Then by Theorem \ref{thm1},
we have $d_1 (K_{2, 4k \pm 1}) = -2k$ for any $k \in \mathbb{N}$.
Furthermore, it follows from \cite[Theorem 1.1]{hedden} and \cite{hom}
that $\tau (K) = -g(K)$ and $\varepsilon = -1$,
where $g(K)$ denotes the genus of $K$.
Hence by Theorem \ref{thm.hom}, we have
$$
\tau(K_{2,4k+1}) = 2\tau(K) + 2k + 1
$$
and
$$
\tau(K_{2,4k-1}) = 2\tau(K) + 2k.
$$
For any $m,n \in \mathbb{N}$ and the knots $\{ K_i \}_{i \in \mathbb{Z}}$ in Figure \ref{K_n},
we define
$$
K^{m,n} := K_{n} \# K_{n+1} \# \cdots \# K_{n+m-1},
$$
where $\#$ denotes the connected sum.
Note that $K^{m,n}$ is a negative knot and has $g(K^{m,n}) = m$ 
for any $m,n \in \mathbb{N}$.
Hence for any two even integers $a$ and $b$ with $a > b \geq 0$,
if $b/2$ is odd, then we take $\{ (K^{(2a-b+2)/4,n})_{2,2a+1} \}_{n \in \mathbb{N}}$
and we have
$$
|d_1((K^{(2a-b+2)/4,n})_{2,2a+1})|= a
$$
and
$$
2|\tau((K^{(2a-b+2)/4,n})_{2,2a+1})| = 2|(-a+ \frac{b}{2} -1) + a + 1|= b.
$$
Otherwise, we take $\{ (K^{(2a-b)/4,n})_{2,2a-1} \}_{n \in \mathbb{N}}$
and we have
$$
|d_1((K^{(2a-b)/4,n})_{2,2a-1})|= a
$$
and
$$
2|\tau((K^{(2a-b)/4,n})_{2,2a-1})| = 2|(-a+ \frac{b}{2}) + a|= b.
$$

It is easy to verify
that if $l \neq n$, then $(K^{m,l})_{2,q}$ is not concordant to $(K^{m,n})_{2,q}$
for any $m \in \mathbb{N}$ and any odd integer $q > 1$.
Actually, the Alexander polynomial 
$\Delta_{(K^{m,l})_{2,q}\# (K^{m,n})_{2,q}}(t) 
= \Delta_{K^{m,l}}(t^2) \cdot \Delta_{K^{m,n}}(t^2) \cdot (\Delta_{T_{2,q}}(t))^2$
is not of the form $f(t)f(t^{-1})$. This completes the proof.
\end{proof}

\section{Proof of Proposition \ref{prop}}
In this last section, we prove Proposition \ref{prop}.
In order to prove it, we observe
the normalized Alexander polynomial of cable knots,
while we gave geometrical observations  in the other sections.

We say that a Laurent polynomial $f(t)$ is {\it symmetric} 
if $f(t)$ satisfies $f(t) = f(t^{-1})$.
Any symmetric Laurent polynomial $f(t)$ is the form of
$$
f(t) = a_0(f) + \sum_{i =1}^{d} a_i(f)(t^i + t^{-i}).
$$
We denote $d$ by $\deg f$.
Furthermore, the set of symmetric Laurent polynomials,
denoted by $S$, is a $\mathbb{Z}$-submodule of $\mathbb{Z}[t, t^{-1}]$. 
we define $t_0(f)$ as 
$$
t_0 (f):= \sum_{i=1}^{\deg f} i a_i(f).
$$
Since $a_i$ is a homomorphism from $S$ to $\mathbb{Z}$ for any $i$,
$t_0$ is also a homomorphism.
The following proposition is derived from  \cite[Proposition 8.1]{ozsvath-szabo}.
\begin{prop}
\label{prop2}
Let $K$ be a knot in $S^3$ such that $S^3_p(K)$ is a lens space for some $p \in \mathbb{N}$,
and $\Delta_K(t)$ the normalized Alexander polynomial of $K$. Then
$$
d_1(K) = -2t_0(\Delta_K).
$$ 
\end{prop}
It is shown in \cite[Theorem 1]{bleiler-litherland} 
that $S^3_{4pq \pm 1}((T_{p,q})_{2, 2pq \pm 1}))$ is a lens space,
and hence we only need to compute $t_0((T_{p,q})_{2, 2pq \pm 1})$
to prove Proposition \ref{prop}.
In order to compute the value, we first prove the following lemma.
Here we denote $t^i + t^{-i}$ by $T_i$.
\begin{lem}
\label{lem3}
For any symmetric Laurent polynomial $f(t)$ with $f(1)=1$
and any integer $k \geq 1$, we have
$$
t_0(f(t) \cdot T_k) = 
\left\{ 
\begin{array}{ll}
k &( k \geq \deg f)\\
k a_0(f) + \sum_{i =1}^{k} 2k a_i(f) + \sum_{i =k+1}^{\deg f} 2i a_i(f)
& (1 \leq k < \deg f)
\end{array}
\right.
.
$$
\end{lem}  
\def\proofname{Proof}
\begin{proof}
Note that $T_i \cdot T_k = (t^i +t^{-i})(t^k +t^{-k}) = T_{i+k} + T_{i-k}$.
This equality gives 
$$
f(t) \cdot T_k = a_0(f) T_k + \sum_{i =1}^{\deg f} a_i(f)T_{i+k} 
+ \sum_{i =1}^{\deg f} a_i(f) T_{i-k}
$$
and
$$
t_0(f(t) \cdot T_k) = k a_0(f) + \sum_{i =1}^{\deg f} (i+k + |i-k|) a_i(f).
$$
If $k \geq \deg f$, then we have
\begin{eqnarray*}
t_0(f(t) \cdot T_k) &=& k a_0(f) + \sum_{i =1}^{\deg f} (i+k + k-i) a_i(f)\\
\ &=&k (a_0(f) + \sum_{i =1}^{\deg f} 2a_i(f))\\
\ &=& k f(1).
\end{eqnarray*}
Since $f(1)=1$, this gives the desired equality.
It is easy to check that the equality for  the case where $1 \leq k < \deg f$.
\end{proof}

\def\proofname{Proof of Proposition \ref{prop}}
\begin{proof}
We note that for any positive odd integer $r$,
\begin{eqnarray*}
\Delta_{(T_{p,q})_{2,2r+1}}(t) &=& \Delta_{(T_{p,q})}(t^2) \cdot \Delta_{T_{2,2r+1}}(t)\\
\ &=& (-1)^r \Delta_{(T_{p,q})}(t^2) \cdot (1 + \sum_{k=1}^r (-1)^k T_k).
\end{eqnarray*}
Hence if we set $t'_0 := (-1)^r t_0(\Delta_{(T_{p,q})_{2,2r+1}}(t))$,
then we have 
\begin{equation}
\label{eq2}
t'_0 = t_0(\Delta_{T_{p,q}}(t^2)) +
\sum_{k=1}^r (-1)^k t_0(\Delta_{T_{p,q}}(t^2) \cdot T_k).
\end{equation}
We suppose that $r > \deg \Delta_{T_{p,q}}(t^2) =: d'$ and we
set $a'_i := a_i (\Delta_{(T_{p,q})}(t^2))$ for $0 \leq i \leq d'$.Then
it follows from Lemma \ref{lem3} that
$$
t_0(\Delta_{(T_{p,q})}(t^2)) =  \sum_{i=1}^{d'}i a'_i
$$
and
$$
t_0(\Delta_{(T_{p,q})}(t^2) \cdot T_k) =
\left\{ 
\begin{array}{ll}
k &( k \geq d')\\
k a'_0 + \sum_{i =1}^{k} 2k a'_i + \sum_{i =k+1}^{d'} 2i a'_i
& (1 \leq k < d')
\end{array}
\right.
.
$$
These equalities reduce (\ref{eq2}) to
$$
t'_0 = 
\left\{ 
\begin{array}{ll}
(r/2)a'_0 + \sum_{i=1}^{d'/2}(r+1)a'_{2i-1} + \sum_{i=1}^{d'/2}r a'_{2i} &(r : \text{ even})\\
-\{((r+1)/2)a'_0 + \sum_{i=1}^{d'/2}r a'_{2i-1} + \sum_{i=1}^{d'/2}(r+1) a'_{2i}\}
& (r : \text{ odd})
\end{array}
\right.
$$
(note that $d' = \deg \Delta_{T_{p,q}}(t^2) = 2 \deg \Delta_{T_{p,q}}(t)$).
Furthermore, we note that
$$
a'_i =
\left\{
\begin{array}{ll}
a_{i/2}(\Delta_{T_{p,q}}(t)) & (i : \text{ even})\\
0 & (i : \text{ odd}) 
\end{array}
\right.
.
$$
Thus we have
\begin{eqnarray*}
t'_0 &=&
\left\{
\begin{array}{ll}
(r/2)\{ a_0(\Delta_{T_{p,q}}(t)) + \sum_{i=1}^{d'/2}2 a_{i}(\Delta_{T_{p,q}}(t))\} 
&(r : \text{ even})\\
-((r+1)/2) \{ a_0(\Delta_{T_{p,q}}(t)) + \sum_{i=1}^{d'/2}2 a_{i}(\Delta_{T_{p,q}}(t)) \}
& (r : \text{ odd})
\end{array}
\right.
\\
\ &=& (-1)^r \Big\lceil \frac{r}{2} \Big\rceil\\
\ &=& \frac{-1}{2} \cdot (-1)^r d_1(T_{2,2r+1}).
\end{eqnarray*}

This implies that for $r > \deg \Delta_{T_{p,q}}(t^2)$, we have
$$-2 t_0(\Delta_{(T_{p,q})_{2,2r +1}}(t)) = d_1(T_{2,2r + 1}).$$
In particular, $(2pq \pm 1-1)/2 > \deg \Delta_{T_{p,q}}(t^2)$, and hence we have
$$
d_1 ((T_{p,q})_{2,2pq \pm 1}) = -2t_0(\Delta_{(T_{p,q})_{2,2pq \pm 1}}(t))
= d_1(T_{2,2pq \pm 1}).
$$
\end{proof}

\end{document}